\documentclass[journal]{IEEEtran}

\usepackage{amsfonts}
\usepackage{graphicx}
\usepackage{amssymb}
\usepackage{epstopdf}
\usepackage{amsmath,amsthm}
\usepackage{enumerate}
\usepackage{eufrak}
\usepackage{cite}
\usepackage{mathcomp}
\usepackage{supertabular}
\usepackage{longtable}
\usepackage{stmaryrd}
\usepackage{url}
\usepackage{cases}

\usepackage{color}
\usepackage{rotating}
\usepackage{float}

\newtheorem{theorem}{Theorem}
\newtheorem{lemma}{Lemma}
\newtheorem{proposition}{Proposition}

\theoremstyle{definition}
\newtheorem{definition}{Definition}
\newtheorem{example}{Example}

\newtheorem{conjecture}{Conjecture}

\theoremstyle{remark}

\newcommand{\A}{\mathcal A}
\newcommand{\B}{\mathcal B}
\newcommand{\C}{\mathcal C}
\newcommand{\G}{\mathcal G}
\newcommand{\vu}{\sf u}
\newcommand{\pr}{{\rm Prob}}
\newcommand{\bz}{\mathbb{Z}}
\newcommand{\gf}{\mathbb{F}}

\begin{document}

\sloppy

\title{Complexity of Dependencies in Bounded Domains,
Armstrong Codes, and Generalizations}

 \author{
   \IEEEauthorblockN{Yeow Meng Chee, Hui Zhang, and Xiande Zhang}
   \thanks{Research of the authors was supported in part by the Singapore
National Research Foundation under Research Grant NRF-
CRP2-2007-03.}
   \thanks{Yeow Meng Chee, Hui Zhang and Xiande Zhang are with the School of Physical and Mathematical Sciences, Nanyang Technological
   University, 21 Nanyang Link, Singapore 637371,
   Singapore (email: \{ymchee, huizhang, xiandezhang\}@ntu.edu.sg).}
\thanks{This paper was presented
in part at the 2013 IEEE International Symposium on Information
Theory \cite{CheeZhang:2013}.} 
\thanks{Copyright (c) 2014 IEEE. Personal use of this material is permitted.  However, permission to use this material for any other purposes must be obtained from the IEEE by sending a request to pubs-permissions@ieee.org.}} 

\maketitle

\begin{abstract}
The study of Armstrong codes is motivated by the problem of
understanding complexities of dependencies in relational database
systems, where attributes have bounded domains. A
$(q,k,n)$-Armstrong code is a $q$-ary code of length $n$ with
minimum Hamming distance $n-k+1$, and for any set of $k-1$
coordinates there exist two codewords that agree exactly there. Let
$f(q,k)$ be the maximum $n$ for which such a code exists. In this
paper, $f(q,3)=3q-1$ is determined for all $q\geq 5$ with three
possible exceptions. This disproves a conjecture of Sali. Further,
we introduce generalized Armstrong codes for branching, or
$(s,t)$-dependencies, construct several classes of optimal Armstrong
codes and establish lower bounds for the maximum length $n$ in this
more general setting.
\end{abstract}

\section{Introduction}
Let $A$ be a set of $n$ {\em attributes}. Each attribute $x\in A$ is
associated with a set $\Omega_x$, called its {\em domain}. A {\em relation}
is a finite set $R$ of $n$-tuples (called {\em data items}), such that
$R\subseteq \times_{x\in A}\Omega_x$.
A relational database {\em table} is an $m\times n$ array where each column
is indexed by an  attribute and each row corresponds to a data item in $R$.
We denote this table by $R(A)$.
More specifically, if $R=\{(d_{i,x})_{x\in A}: 1\leq i\leq m\}$, then the cell in
$R(A)$ with row index $i$ and column index $x$ has entry $d_{i,x}$.
A {\em relational database} is a set
of tables, where different tables may be defined over different
attribute sets.

For a given table $R(A)$ and $X\subseteq A$, the {\em $X$-value} of
a data item $d=(d_x)_{x\in A}$ in $R(A)$ is the $|X|$-tuple
$d\!\mid_X=(d_x)_{x\in X}$. Let $X\subseteq A$ and $y\in A$ for a
given table $R(A)$. We say that $y$ {\em (functionally) depends} on
$X$, written $X\rightarrow y$, if no two rows of $R(A)$ agree in $X$
but differ in $y$. In other words, if the $X$-value of a data item
is known, then its $\{y\}$-value can be determined with certainty. A
{\em key} for $R(A)$ is a subset $K\subseteq A$, such that
$K\rightarrow b$ for all $b\in A$. A key $K$ is called {\em minimal}
if no subset of $K$ is a key.

Identifying functional dependencies, especially key dependencies, is
important in relational database design
\cite{Armstrong:1974,Bernstein:1976,Beerietal:1977,Rissanen:1977}.
From the schema design point of view, the question of whether a
given collection $\Sigma$ of functional dependencies has an {\em
Armstrong instance} for $\Sigma$, that is, a table that satisfies a
functional dependency $X\rightarrow y$ if and only if $X\rightarrow
y$ is implied by $\Sigma$, is well studied. The existence of an Armstrong
instance for any given set of functional dependencies was proved by
Armstrong \cite{Armstrong:1974} and Demetrovics
\cite{Demetrovics:1979}. Further investigations (see for example,
\cite{Demetrovicsetal:1992}) concentrated on the minimum size of an
Armstrong instance, since it is a good measure of the complexity of
a set of functional dependencies, or a set of minimal keys.


Earlier work on Armstrong instances were mostly studied by
assuming that the domain of each attribute is countably
infinite. Recently,   the study of {\em higher order data model} in
\cite{Hartmannetal:2004,Sali:2004} considered the question of Armstrong
instances with bounded domains.  Another reason for considering
bounded domains is that for many attributes, their domains are
well defined finite sets. For example, the
age of a person can take values from the set
$\{0,1,\ldots,130\}$.

Thalheim \cite{Thalheim:1992} investigated the
maximum number of minimal keys in the case of bounded domains and
showed that restrictions on the sizes of domains make
significant differences. It is natural to ask what one can say about
Armstrong instances if all attributes have domains restricted to size
$q$. Let $K_n^k$ denote
the collection of all $k$-subsets of an $n$-element attribute set
$A$.

\begin{definition}\label{fqk} Let $q,k > 1$ be integers.
Let $f(q, k)$ denote the maximum $n$ such that there exists an Armstrong
instance for $K_n^k$ being the system of minimal keys.
\end{definition}

The problem of determining $f(q,k)$ was introduced in
\cite{SaliSchewe:2008} and investigated in
\cite{Katonaetal:2008,SaliSzekely:2008}. The only known values of
$f(q, k)$ are $ f(q, 2) =\binom{q+1}{2}$, which were determined in
\cite{Katonaetal:2008}.

One of the main contributions of this paper is the determination of
$f(q,3)$. We prove that $f(q,3)=3q-1$ for all $q\geq 5$, except
possibly for $q\in \{14,16,20\}$. This disproves a conjecture of
Sali \cite{Sali:2011}.

When functional dependencies
are not known, the concept was
generalized to $(s,t)$-dependencies to improve
storage efficiency
\cite{Demetrovicsetal:1992,Demetrovicsetal:1995,Demetrovicsetal:1998,KatonaSali:2004}. In this paper, we introduce the analogous problem of
determining $f(q,k)$ when extended to $(s,t)$-dependencies, that is the function $f_{s,t}(q,k)$ (see Section IV for detailed definition). We show that
$f_{1,t}(q,2)=\binom{qt+1}{t+1}$, $f_{2,2}(q,4)=2q-1$ and establish
several lower bounds of $f_{1,t}(q,k)$ by constructive method and
probabilistic method.

\section{Preliminaries}
Throughout this paper, we view an $m\times n$ Armstrong instance with domains of size $q$
as a $q$-ary code $C$ of length $n$ and size $m$, where the
codewords are precisely the rows of the instance.

For a positive integer $k$, $[k]$ denotes the set
of integers $\{1,2,\ldots, k\}$ and $\bz_k$ denotes the ring of integers modulo $k$. For any subset $S\subset \bz_k$, let $aS\triangleq\{as:s\in S\}$ and $a+S\triangleq\{a+s:s\in S\}$.

\subsection{Armstrong Codes}

Katona et al. \cite{Katonaetal:2008} characterized the $q$-ary code
$C$ corresponding to an Armstrong instance with $K_n^k$ as the set
of minimal keys, as follows:

\begin{enumerate}[(i)]
\item $C$ has minimum Hamming distance at least $n - k + 1$;
\item for
any set of $k - 1$ coordinates there exist two codewords agreeing in
exactly those coordinates.
\end{enumerate}

A $(k-1)$-set of coordinates can be considered as a ``direction'', so
in $C$ the minimum distance is {\em attained in all directions}.
Such a code $C$ is called an {\em Armstrong code},
or more precisely, a {\em $(q, k, n)$-Armstrong code}. It is obvious that
$f(q,k)$ is the maximum $n$ such that there exists a $(q,
k, n)$-Armstrong code. The following bounds on $f(q,k)$ are known.

\begin{theorem}[Blokhuis et al. \cite{BBS2014}, Katona et al. \cite{Katonaetal:2008}, Sali and Sz\'ekely \cite{SaliSzekely:2008}]
\hfill
\begin{enumerate}[(i)]
\item Let $q>4$. Then $f(q,k)\geq \left\lceil\frac{k}{2} \log q-1\right\rceil$ for all
sufficiently large $k$.
\item  $f(2,k)\geq k+3$ for all $k\geq7$. Further, there exists a constant $c>1$ such that $f(2,k)\geq \lfloor ck\rfloor$ for all
sufficiently large $k$.
\item Let $q > 1$ and $k > 2$.
Then
\begin{equation}
f(q, k) \leq
q(k-1)\left(1+\frac{q-1}{\sqrt{\frac{2(qk-q-k+2)^{k-1}}{(k-1)!}}-q}\right).
\label{ub}
\end{equation}
\item If $q\geq 2$ and $k\geq 5$, then the bound (\ref{ub})
can be improved to $f(q, k) \leq  q(k -1)$, except when
$(k, q) \in\{ (5, 2), (5, 3), (5, 4), (5, 5), (6,
2)\}$.
\item For fixed $q>1$, we have
\begin{equation*}
\frac{\sqrt{q}}{e}k<f(q,k)<(q-\log q)k
\end{equation*}
for all sufficiently large $k$.
\end{enumerate}
\end{theorem}

\vskip 5pt

\begin{proposition}[Katona et al. \cite{Katonaetal:2008}]  \label{3q-1ub}
For $q>1$, $f(q, 3)\leq 3q -1$.
\end{proposition}

\subsection{Orthogonal Double Covers}

The concept of orthogonal double covers originates in  conjectures
of Demetrovics et al. \cite{Demetrovicsetal:1985} concerning database
constraints and was formalized later by Ganter et al.
\cite{Ganteretal:1994}. Let $X$ be a finite set.
A partition of $X$ is said to {\em cover}
$T\subseteq X$ if $T$ is contained in some part of the partition.
Let $K_m$ denote a complete graph  on $m$ vertices.
For convenience, let $a_1K_{m_1}\cup \cdots \cup a_sK_{m_s}$ denote the disjoint union of $a_i$ copies of $K_{m_i}$, $1\leq i\leq s$.

\begin{definition} Let $X$ be a set of size $m$.
A set of partitions of $X$ is called an {\em orthogonal double
cover} (ODC) of $K_m$ (with its vertices identified with elements of
$X$) if it satisfies the following properties:
\begin{enumerate}[(i)]
\item for any two partitions, there is exactly one 2-subset of $X$
that is covered by both partitions;
\item each 2-subset of $X$ is
covered by exactly two different partitions.
\end{enumerate}
\end{definition}

A construction of $(q,3,n)$-Armstrong codes from ODC's was introduced by Sali in \cite{Sali:2011}. View each part of a partition of $X$ as a complete subgraph of
$K_m$ over $X$. Then each partition can be regarded as a disjoint
union of complete subgraphs of $K_m$. Note that a part of size one corresponds to a complete subgraph consisting of only one isolated vertex. If an ODC consists of $n$
partitions, each of which is isomorphic to a graph $G$, then we say
the ODC is an {\em ODC by $n$ $G$'s}. Suppose that $G$ is a disjoint
union of $q$ complete subgraphs, then an ODC of $K_m$ by $n$ $G$'s
gives an $m\times n$ Armstrong instance over $[q]$ as follows. For
each partition of the ODC, arbitrarily order the $q$ parts, and
construct a column $\vu$ of length $m$, with coordinates indexed by
elements of $X$, such that for $i\in X$, ${\vu}_i=j$ if and only if
$i$ is contained in the $j$-th part of the partition. It is easy to
check that the set of rows of this Armstrong instance is a
$(q,3,n)$-Armstrong code.

\begin{example}\label{k7}
In \cite{Demetrovicsetal:1985}, there is an ODC of $K_7$ by
seven $2K_3\cup K_1$'s over $\bz_7$ with each partition $P_i$, $i\in \bz_7$ consisting of three parts $\{i\}$, $i+\{1,2,4\}$ and $i+\{ 3 , 5 , 6 \}$. Then a $(3,3,7)$-Armstrong code is
constructed as below.

$$\begin{array}{ccccccc}
3 & 2 & 2 & 1 & 2 & 1 & 1 \\
1 & 3 & 2 & 2 & 1 & 2 & 1 \\
1 & 1 & 3 & 2 & 2 & 1 & 2 \\
2 & 1 & 1 & 3 & 2 & 2 & 1 \\
1 & 2 & 1 & 1 & 3 & 2 & 2 \\
2 & 1 & 2 & 1 & 1 & 3 & 2 \\
2 & 2 & 1 & 2 & 1 & 1 & 3 \\
\end{array}$$
\end{example}

Ganter and Gronau \cite{GanterGronau:1991} proved that for $q\geq 5$,
there exists an
ODC of $K_{3q-2}$ by $3q-2$ $(q-1)K_3\cup K_1$'s, settling a conjecture
of Demetrovics et al. \cite{Demetrovicsetal:1985}. This result also implies the
existence of a $(q,3,3q-2)$-Armstrong code. Hence, we have $f(q,3)\geq 3q-2$.
Furthermore, it is easy to show that $f(2,3)=4$.
This led Sali \cite{Sali:2011} to make the following conjecture.

\begin{conjecture}[Sali \cite{Sali:2011}]
\label{wrong} For all $q\geq 2$, $f(q, 3) = 3q -2$.
\end{conjecture}

Unfortunately, this conjecture is false. When $m\geq 2$, an ODC of
$K_{6m+2}$ by $6m+2$ $2mK_3\cup K_2$'s has been constructed by
Gronau et al. \cite{Gronauetal:1995b}. This gives a
$(2m+1,3,6m+2)$-Armstrong code and hence $f(2m+1,3)\geq 6m+2$. Thus,
Conjecture \ref{wrong} is false for all odd $q\geq 5$. One of the
primary aims of this paper is to prove that Conjecture \ref{wrong}
is also false for even $q$. In fact, we determine that $f(q,3)=3q-1$
for all $q\geq5$, with three possible exceptions.

\section{$(q,3,3q-1)$-Armstrong Codes}

We prove $f(q,3)=3q-1$ by showing the existence of $(q,3,3q-1)$-Armstrong codes.
Our proof is constructive and uses techniques from combinatorial design theory. We
briefly review some required concepts below.

\subsection{Combinatorial Designs}

A {\em set system} is a pair ${\frak S}=(X,\A)$, where $X$ is a
finite set of {\em points} and $\A\subseteq 2^X$. Elements of $\A$
are called {\em blocks}. The {\em order} of $\frak S$ is the number
of points in $X$, and the {\em size} of $\frak S$ is the number of
blocks in $\A$. Let $K$ be a set of positive integers. A set system
$(X,\A)$ is {\em $K$-uniform} if $|A|\in K$ for all $A\in\A$. A {\em
parallel class} of a set system $(X,\A)$ is a set ${\cal
P}\subseteq\A$ that partitions $X$. A {\em resolvable set system} is
a set system whose set of blocks can be partitioned into parallel
classes.

\begin{definition}
A triple system TS$(m, \lambda)$ is a $\{3\}$-uniform set system
$(X, \A)$ of order $m$ such that every $2$-subset of $X$ is
contained in exactly $\lambda$ blocks of $\A$.
\end{definition}

\begin{definition}
Let $(X,\A)$ be a set system and let $\G$ be a partition of $X$ into
subsets, called {\em groups}. The triple $(X,\G,\A)$ is a {\em group
divisible design} (GDD) when every 2-subset of $X$ not contained in a
group is contained in exactly one block, and $|A\cap G|\leq 1$ for all
$A\in\A$ and $G\in\G$.

\end{definition}

We denote a GDD $(X,\G,\A)$ by $k$-GDD if $(X,\A)$ is $\{k\}$-uniform.
The {\em type} of a GDD $(X,\G,\A)$ is the multiset $\langle |G| :
G\in\G\rangle$. When more convenient, the exponential notation is
used to describe the type of a GDD: a GDD of type $g_1^{t_1}
g_2^{t_2} \cdots g_s^{t_s}$ is a GDD where there are exactly $t_i$
groups of size $g_i$, $i\in[s]$. The following results are known
(see, for example, \cite{Abeletal:2007b,Ge:2007}).

\begin{theorem}
\hfill
\label{gdd}
\begin{enumerate}[(i)]
\item A resolvable TS$(m,2)$ exists if and only if
$m\equiv 0 \pmod 3$ and $m\neq 6$.
\item There exists a $4$-GDD of type $2^um^1$
for each $u \geq 6$, $u \equiv 0 \pmod 3$ and $m \equiv 2 \pmod 3$
with $2 \leq m \leq u - 1$, except for $(u, m) = (6, 5)$ and possibly
except for $(u, m) \in \{(21, 17), (33, 23), (33, 29), (39, 35), (57,
44)\}$.
\end{enumerate}
\end{theorem}

\subsection{Extorthogonal Double Covers}

A {\em suborthogonal double cover} (subODC) is a collection of
partitions of $[m]$ similar to an ODC except that for any two
partitions there is {\em at most} one 2-subset of $[m]$ covered by
both partitions. SubODCs were first studied by Hartmann and
Schumacher \cite{HartmannSchumacher:2000}, who considered them as
generalized ODCs under circumstances when ODCs do not exist. Here,
we consider another generalization, called {\em extorthogonal double
covers} (extODC). These are similar to ODCs, except that for any two
partitions there is {\em at least} one 2-subset of $[m]$ covered by
both partitions. We construct  $(q,3,3q-1)$-Armstrong codes from a
special class of extODCs of $K_{3q}$ by $qK_3$'s.

\begin{proposition}\label{equi}
If there exists an extODC of $K_{3q}$ by $qK_3$'s, then
$f(q,3)=3q-1$.
\end{proposition}

\begin{proof}
By considering 2-subsets, the number of partitions in an extODC of
$K_{3q}$ by $qK_3$'s is easily seen to be
${2\binom{3q}{2}}/{3q}=3q-1$. For each partition, arbitrarily order
the $q$ parts. Define a $3q\times (3q-1)$  $q$-ary array by indexing
each column by a partition and each row by a point of the extODC.
For each partition, the corresponding column has the symbol $i$ in
the rows indexed by the points in the $i$th part. The set of rows in
this array is a $(q,3,3q-1)$-Armstrong code, by the definition of an
extODC. This, together with Proposition \ref{3q-1ub}, implies that
$f(q,3)=3q-1$.
\end{proof}
\vskip 10pt

It is easy to see that  an extODC of $K_{3q}$ by $3q-1$ $qK_3$'s is
a resolvable TS$(3q,2)$ with the additional property that every two
parallel classes cover a common 2-subset.  Although $f(q,3)$ is
known for odd $q$, it is still interesting to know when extODCs of
$K_m$, $m$ odd, can exist. We have the following result for $m=3q$,
$q$ odd.

\begin{proposition}\label{eodcodd}
There exists an extODC of $K_{3q}$ by $qK_3$'s, for all odd
$q\geq 5$.
\end{proposition}

\begin{proof}
Let $u=(3q-1)/2$.  Starting from a $4$-GDD $(X,\G,\A)$ of type
$2^u$, whose existence is guaranteed by Theorem~\ref{gdd}, we
construct an extODC of $K_{2u+1}$ over $X\cup\{\infty\}$ from $\A$.
For each $x\in X$, let $\B_x=\{B\setminus \{x\}: x\in B\in
\A\}\cup\{G\cup\{\infty\}:x\in G\in \G\}$. Then $\B_x$ is a
partition of $X\cup\{\infty\}$. We claim that $\{\B_x:x\in X\}$ is
an extODC of $K_{2u+1}$.

Indeed, for any two partitions, say $\B_x$ and $\B_y$, both of which cover
$\{x,y\}$ if $x,y$ are in the same group, and cover $B\setminus
\{x,y\}$ if $x,y \in B$ are in distinct groups. For
 each pair $\{x,y\}\subset X\cup\{\infty\}$, if $\{x,y\}\subset
G\cup\{\infty\}$ for some $G\in \G$, then $\{x,y\}$ is covered by
two partitions $\B_g$, $g\in G$; if $x,y \in X$ are in distinct
groups, then there exists exactly one block $B\in \A$ such that
$\{x,y\}\subset B$, while $\{x,y\}$ is covered by two partitions
$\B_g$, $g\in B\setminus \{x,y\}$. Hence, $\{\B_x:x\in X\}$ is an
extODC.
\end{proof}
\vskip 10pt

We now construct extODCs of $K_m$, where $m=3q$ is even. Define a
{\em base partition}  of order $m$, which is a  partition $P$ of
$\mathbb{Z}_{m-1}\cup\{\infty\}$ into triples with the following two
properties:
\begin{enumerate}[(i)]
\item $\langle \pm (a-b): \{a,b\}\subset C\in P \text{ and }  \infty\not\in \{a,b\}\rangle=2(\mathbb{Z}_{m-1}\setminus
\{0\})$.
\item $\langle i: \{a,b\}+i=\{c,d\} \text{ for some } \{a,b\}\subset C, \{c,d\}\subset C' \text{ and } C,C'\in
P\rangle \supset (\mathbb{Z}_{m-1}\setminus \{0\})$, where $\infty
+i:= \infty$.
\end{enumerate}
Here we use angled brackets $\langle \cdot\rangle$ for multisets. For each $j\in \mathbb{Z}_{m-1}$, let $P_j=\{j+C|C\in P\}$. Then $P_j$, $j\in \mathbb{Z}_{m-1}$ are partitions of
$\mathbb{Z}_{m-1}\cup\{\infty\}$, which forms an extODC of $K_m$.
The first property ensures that each pair occurs exactly twice,
while the second ensures that any two partitions cover at least one
common 2-subset.

\begin{proposition} \label{small}
There exists an extODC of $K_{3q}$ by $qK_3$'s, for $q\in \{6,8,10,12\}$.
\end{proposition}

\begin{proof}
The base partitions for extODC of $K_{3q}$, for $q\in\{6,8,10,12\}$,
are given in Table~\ref{basepartition}.


 \begin{table}
 \caption{Base Partitions for Some Small extODCs}
  \label{basepartition}
  \[ \begin{array}{rl}
   \hline
   q & \text{triples} \\
   \hline
6 &\{0,1,2\} \{3,7,12\} \{4,15,\infty\}
\{5,8,14\} \{6,10,13\} \{9,11,16\} \\
 \hline
 8 & \{0,1,2\}
\{3,5,8\} \{4,12,18\} \{6,15,19\} \{7,14,\infty\} \{9,17,21\}\\
&\{10,13,20\} \{11,16,22\} \\
 \hline
10 & \{0,1,2\} \{3,5,8\} \{4,10,20\} \{6,23,\infty\} \{7,11,22\} \{9,17,27\}\\& \{12,16,25\} \{13,21,28\} \{14,19,26\} \{15,18,24\}\\
\hline 12 & \{0,1,2\} \{3,5,8\} \{4,7,15\} \{6,19,34\} \{9,13,27\}
\{10,20,25\} \\&\{11,22,28\} \{12,24,33\} \{14,23,30\} \{16,26,32\}
\{17,21,29\}\\& \{18,31,\infty\}\\  \hline
   \end{array}\]
 \end{table}
\end{proof}

\begin{proposition}
\label{eodceven1} There exists an extODC of $K_{3q}$ by $qK_3$'s,
for all even $q\geq 18$, $q\neq 20$.
\end{proposition}

\begin{proof}
Let $u=(3q-18)/2$. There exists a $4$-GDD $(X,\G,\A)$ of type
$2^u17^1$ by Theorem~\ref{gdd}. We construct an extODC of $K_{3q}$
(on $X'=X\cup\{\infty\}$) from $\A$. Let $G_0$ be the long group in
$\G$ of size $17$. By Proposition~\ref{small}, there exists an
extODC of $K_{18}$ (on $G_0\cup \{\infty\}$) by $17$ $6K_3$'s over
$G_0\cup \{\infty\}$. Let the set of partitions be $\{\C_x:x\in
G_0\}$. For each $x\in (X\setminus G_0)$, let $\B_x=\{B\setminus
\{x\}: x\in B\in \A\}\cup\{G\cup\{\infty\}:x\in G\in \G\}$. For each
$x\in G_0$, let $\B_x=\{B\setminus \{x\}: x\in B\in \A\}\cup \C_x$.
There are $3q-1$ $\B_x$'s in total and each $\B_x$ is a partition of
$X'$. We claim that the set of all $\B_x$'s is an extODC.

Indeed, for any
two partitions $\B_x$ and $\B_y$, they both cover $\{x,y\}$ if $x,y$
are in the same group of size $2$; cover a common 2-subset if $x,y \in
G_0$ since $\C_x$ and $ \C_y$ have a common 2-subset, and both cover
$B\setminus \{x,y\}$ if $x,y \in B$ are in distinct groups. For each
pair $\{x,y\}\subset X'$, if $\{x,y\}\subset G\cup\{\infty\}$ for
some $G\neq G_0$, then $\{x,y\}$ is covered in two partitions $\B_g$,
$g\in G$. If $\{x,y\}\subset G_0\cup\{\infty\}$,  then $\{x,y\}$ is
covered by both $\B_u$ and $\B_v$, where $\{x,y\}$ is contained in  $\C_u$
and $\C_v$. If $x,y$ are in distinct groups, then there exists
exactly one block $B\in \A$ such that $\{x,y\}\subset B$, while
$\{x,y\}$ occurs in $\B_g$, $g\in B\setminus \{x,y\}$. Hence,
$\{\B_x: x\in X\}$ is an extODC of $K_{3q}$.
\end{proof}
\vskip 10pt

Combining Propositions~\ref{equi},~\ref{eodcodd} and
\ref{eodceven1}, we give the main result of this section.

\begin{theorem}
\label{eodceven}
For all $q\geq 5$ and $q\neq 14,16,20$, there
exists an extODC of  $K_{3q}$ by $qK_3$'s, and consequently $f(q,3)=3q-1$.
\end{theorem}

Before closing this section, we estimate the values of $f(q,3)$ for
$q=3$ and $4$.

\begin{proposition}
\label{q34} $f(3,3)=7$ and $f(4,3)\in \{10,11\}$.
\end{proposition}
\begin{proof}For $q=3$, we have $f(3,3)\leq 8$. Suppose that $C$ is
a $(3,3,8)$-Armstrong code of size $m$. We consider the total number $s$ of pairs of
equal entries in the same coordinates of $C$. In any pair of codewords of $C$  at most two coordinates can have equal entries because the  minimum distance of $C$ is six. Hence, $s\leq 2\times\binom{m}{2}$. By the defining condition (ii)
of Armstrong codes, for each pair of coordinates there is at least one pair of codewords agreeing in exactly those coordinates. Further, for different pairs of coordinates, the pairs of codewords are different, i.e., the pairs of equal entries are all different.  Thus $s\geq 2\times\binom{8}{2}$ and then $m\geq 8$. However, since $C$ is also a
ternary code of distance six, we have $m\leq 9$ \cite{Brouwer:1998}.
Now we claim that either $m=8$ or $9$ is impossible, hence $C$ does
not exist. Consider $C$ as an $m\times 8$ array. When $m=8$ or $9$,  each column has at least $7$ or $9$ pairs of equal entries, respectively, which is achieved when the three symbols occur almost the same frequency. Since the total number $s$ of such pairs in $C$ is at most $56$ if $m=8$ or $72$ if $m=9$,  there are exactly $7$ or $9$ such pairs in each column when $m=8$ or $9$, respectively. When $m=8$, $C$ is equivalent to an ODC of $K_8$ by eight $2K_3\cup
K_2$'s, which does not exist by \cite{GanterGronau:1991}. When
$m=9$, $C$ is equivalent to an extODC of $K_9$ by eight $3K_3$'s, which
could be excluded easily by computer search. So we conclude that
$f(3,3)\leq 7$. An optimal code exists by Example~\ref{k7}.


For $q=4$, we have $f(4,3)\leq 11$. A $(4,3,10)$-Armstrong code
exists by the existence of an ODC of $K_{10}$ by ten $K_4\cup3K_2$'s
\cite{GanterGronau:1991}.
\end{proof}

\section{Generalized Armstrong Codes}

The concept of functional dependencies was
generalized by Demetrovics, Katona, and Sali \cite{Demetrovicsetal:1992}.

\begin{definition}
Let $X\subseteq A$ and $y\in A$ for a given table $R(A)$. Then for
positive integers $s\leq t$, we say that
$y$ {\em $(s,t)$-depends} on $X$, written $X\stackrel{(s,t)}{\longrightarrow} y$,
if there do not exist $t+1$ data items (rows)
$d_1,d_2,\ldots,d_{t+1}$ of $R(A)$ such that
\begin{enumerate}[(i)]
\item $|\{ d_i\!\mid_{\{x\}}: 1\leq i\leq t+1\}| \leq s$ for each $x\in X$, and
\item $|\{ d_i\!\mid_{\{y\}}: 1\leq i\leq t+1\}| = t+1$.
\end{enumerate}
\end{definition}

 Our usual concept of functional dependency is equivalent to
the special case of $(1,1)$-dependency. When functional dependencies
are not known,
$(s,t)$-dependencies identified
in a relational database can still be exploited for improving
storage efficiency
\cite{Demetrovicsetal:1992,Demetrovicsetal:1995,Demetrovicsetal:1998,KatonaSali:2004}.

Given $1\leq s\leq t$, an {\em $(s,t)$-dependent key} $K$ is a
subset of the attribute set $A$, such that $R(A)$ satisfies
$(s,t)$-dependencies $K\stackrel{(s,t)}{\longrightarrow}  y$ for all
$y\in A$. A key $K$ is called {\em minimal} if no subset of $K$ is
an $(s,t)$-dependent key. Here, we generalize Armstrong codes from
functional dependencies into $(s,t)$-dependencies.

A $q$-ary code $C$ is called a {\em $(q,k,n)_{s,t}$-Armstrong code}
if
\begin{enumerate}[(i)]
 \item for any $t+1$ rows of $C$, there exist at most $k-1$ columns such that each column has at most $s$ distinct elements in the $t+1$ rows, and
\item for any $k-1$ columns of $C$, there exist $t+1$ rows such that each of the $k-1$ columns has at most $s$ distinct elements in the $t+1$
rows. Further, there exists a column having exactly $t+1$ distinct
elements in these $t+1$ rows.
  \end{enumerate}

Consider the Armstrong code defined above as an Armstrong instance
with $n$ attributes. The first property in the definition makes sure
each $k$-subset of $K_n^k$ is an $(s,t)$-dependent key, while the
second property ensures that each key is minimal. It is clear that
we need $q> s, t$ and $k>1$ for a $(q,k,n)_{s,t}$-Armstrong code to
be meaningful. Note that a $(q,k,n)_{1,1}$-Armstrong code is just a
$(q,k,n)$-Armstrong code.

\begin{definition}
Let $q > t\geq s\geq 1$ and $k > 1$. Then $f_{s,t}(q, k)$ denotes
the maximum $n$ such that there exists a $(q,k,n)_{s,t}$-Armstrong
code.
\end{definition}

As with the Armstrong codes for functional dependencies \cite{Katonaetal:2008},
we have the following
restrictions on $(q,k,n)_{s,t}$-Armstrong codes. Let $\phi$ be the
least number of submultisets  $S\subset M$ of size $t+1$ with at
most $s$ distinct elements, where $M$ ranges over all multisets of
size $m$ over $[q]$.

\begin{proposition}\label{condition1}
 Let $C$ be a $(q, k, n)_{s,t}$-Armstrong code and let $m=|C|$. Then $\binom{m}{t+1}\geq
\binom{n}{k-1}$ and $n\cdot \phi \leq (k-1)\binom{m}{t+1}$.
\end{proposition}

\begin{proof} Let $T$ be a set of $k-1$ columns of $C$. By
condition (ii), there exists a set $R_T$ of $t+1$ rows such that
each column of $T$ has at most $s$ distinct elements in $R_T$. By
the first defining condition (i) of a $(q,k,n)_{s,t}$-Armstrong
code, $R_T$ is distinct for distinct $T$. The first inequality then
follows. The second inequality holds by the definition of $\phi$ and
the defining condition (i).
\end{proof}
\vskip 10pt

As in \cite{Katonaetal:2008}, the two inequalities in
Proposition~\ref{condition1} can give two upper bounds of
$f_{s,t}(q, k)$, where one is obviously increasing in $m$ and the other
could be proved to be decreasing in $m$. Thus there is a universal
upper bound of $f_{s,t}(q, k)$ at certain $m$ where the two upper
bounds intersect. However, it is impossible to give an explicit
universal upper bound in most cases. We will use this method
to explore  values of $f_{s,t}(q, k)$ for some special cases.

\subsection{The Case $s=1$ and $k=2$}

\begin{proposition}\label{condition2}
When $s=1$ and $q<m$, we have
\begin{align*}
\phi&\geq
r\binom{h+1}{t+1}+(q-r)\binom{h}{t+1}=q\binom{h}{t+1}+r\binom{h}{t},
\end{align*}
 where $m=qh+r$,
with $0\leq r<q$.
\end{proposition}

\begin{proof}
Similar to the proof in \cite[Lemma 3.2]{Katonaetal:2008}, let $m_1$
and $m_2$ be the number of two distinct symbols in $M$,  where $M$
is a multiset of size $m$ over $[q]$. The inequality follows by the
fact that $\binom{m_1}{t+1}+\binom{m_2}{t+1}\geq
\binom{m_1+1}{t+1}+\binom{m_2-1}{t+1}$ for all $m_1$ and $m_2$
satisfying $m_2-m_1\geq 2$.
\end{proof}

\begin{proposition}\label{decreasing}
The function
$g(m)=\frac{(k-1)\binom{m}{t+1}}{q\binom{h}{t+1}+r\binom{h}{t}}$ is decreasing in $m$, where $h$ and $r$ are functions of $m$ such that $m=qh+r$ with
 $0<r+1\leq q<m$.
\end{proposition}

\begin{proof} We prove that $g(m)\geq g(m+1)$ in two cases.
When $r+1<q$, $m+1=qh+(r+1)$. We
have to verify that
\begin{align*}
\frac{(k-1)\binom{m}{t+1}}{q\binom{h}{t+1}+r\binom{h}{t}}\geq
\frac{(k-1)\binom{m+1}{t+1}}{q\binom{h}{t+1}+(r+1)\binom{h}{t}}.
\end{align*} After carrying out the obvious cancelations, this leads
to $q-r-1\geq 0$ which is trivially true. When $r+1=q$,
$m+1=q(h+1)$, it is easy to check that \begin{align*}
\frac{(k-1)\binom{m}{t+1}}{q\binom{h}{t+1}+r\binom{h}{t}}=
\frac{(k-1)\binom{m+1}{t+1}}{q\binom{h+1}{t+1}},
\end{align*}
i.e., $g(m)=g(m+1)$.
\end{proof}

\begin{proposition}\label{1v}
 $f_{1,t}(q,2)=\binom{qt+1}{t+1}$.
\end{proposition}

\begin{proof} By Proposition~\ref{condition1}, we have  $f_{1,t}(q,2)\leq \binom{m}{t+1}$ and   $f_{1,t}(q,2)\leq \frac{\binom{m}{t+1}}{q\binom{h}{t+1}+r\binom{h}{t}}$,  where $m=qh+r$,
with $0\leq r<q$. Since $\binom{m}{t+1}$ is increasing in $m$ and $\frac{\binom{m}{t+1}}{q\binom{h}{t+1}+r\binom{h}{t}}$ is decreasing in $m$ by Proposition~\ref{decreasing}, the upper bound  $f_{1,t}(q,2)\leq\binom{qt+1}{t+1}$ is the universal upper bound obtained by setting $m=qt+1$ (i.e., $h=t$ and $r=1$) where the two upper bounds intersect. The lower bound is given by construction. Construct a
 $(qt+1) \times \binom{qt+1}{t+1}$ array as follows. For each
column, we have exactly one subset of $t+1$ rows with equal symbols
and all other $q-1$ symbols occurring exactly $t$ times. We do so
such that each column has a distinct subset of $t+1$ rows with equal
symbols. It is clear that this array satisfies the first property of
a $(q, 2,\binom{qt+1}{t+1})_{1,t}$-Armstrong code.

For the second property, any column of the array has $t+1$ rows with
equal symbols. Now for these $t+1$ rows, we need a column having
$t+1$ distinct symbols in these rows. The above array does not have
this property obviously. However, we can slightly rearrange the
symbols occurring $t$ times in each column to satisfy this property.
We do it as follows. Let each column be indexed by the
$(t+1)$-subset of rows which have equal symbols. Let $A$ denote the
set of all the indices and $A'$ be a copy of $A$. Define a bipartite
graph with two parts $A$ and $A'$, two $(t+1)$-subsets are adjacent
if and only if they intersect at most one common symbol. This is a
regular bipartite graph, thus it has a perfect matching $E$. Now for
each edge $\{v,v'\}\in E$, where $v\in A$ and $v'\in A'$, rearrange
the symbols occurring $t$ times in the column $v'$, such that
symbols in the rows of $v$ are all distinct. We can do this since
$|v\cap v'|\leq 1$. This rearrangement will guarantee that for each
$t+1$ rows there is a column having equal symbols and simultaneously
a column having $t+1$ distinct symbols.
\end{proof}

\subsection{The Case $s=t=2$ and $k=4$}

\begin{proposition}\label{condition3}
When $s=t=2$ and $q<m$, we have $\phi\geq
\varphi(m)$, where
\begin{align*}
\varphi(m)= &
r\binom{h+1}{3}+(q-r)\binom{h}{3}+r\binom{h+1}{2}(m-h-1)\\&+(q-r)\binom{h}{2}(m-h).
\end{align*}
Here $h$ and $r$ are functions of $m$ such that $m=qh+r$,
with $0\leq r<q$.
\end{proposition}

\begin{proof}
Let $m_1$ and $m_2$ be the number of two distinct symbols in $M$,
where $M$ is a multiset of size $m$ over $[q]$. The inequality
follows by the fact that
$\binom{m_1}{3}+\binom{m_2}{3}+\binom{m_1}{2}(m-m_1)+\binom{m_2}{2}(m-m_2)\geq
\binom{m_1+1}{3}+\binom{m_2-1}{3}+\binom{m_1+1}{2}(m-m_1-1)+\binom{m_2-1}{2}(m-m_2+1)$
for all $m_1$ and $m_2$ satisfying $m_2-m_1\geq 2$.
\end{proof}

\begin{proposition}\label{decreasing1}
The function
$k(m)=\frac{(k-1)\binom{m}{3}}{\varphi(m)}$ is decreasing in $m$,  where $\varphi(m)$ is defined above.
\end{proposition}

\begin{proof} As in the proof of Proposition~\ref{decreasing}, we
first verify when $m=qh+r$ and $0<r+1<q$ that
\begin{align*}
\frac{(k-1)\binom{m}{3}}{\varphi(m)}\geq
\frac{(k-1)\binom{m+1}{3}}{\varphi(m+1)}.
\end{align*}
Here $m+1=qh+(r+1)$. 
Since there are a large amount of computation, we use Maple to do
the cancelations. This leads to $(q-r-1)(qh+r-2h)\geq 0$ which is
true since $q\geq r+2\geq 2$. If $r+1=q$, then $m=qh+q-1$,
$m+1=q(h+1)$ and
\begin{align*} \varphi(m+1)= &
q\binom{h+1}{3}+q\binom{h+1}{2}(m-h).
\end{align*} We can also check by Maple that
\begin{align*}
\frac{(k-1)\binom{m}{3}}{\varphi(m)}= \frac{(k-1)\binom{m+1}{3}}{\varphi(m+1)},
\end{align*}
i.e., $k(m)=k(m+1)$.
\end{proof}
\vskip 10pt

When $k=4$, we have $f_{2,2}(q,4)\leq m$ and $f_{2,2}(q,4)\leq k(m)$ by Propositions~\ref{condition1} and~\ref{condition3}.
Since $k(m)$ is decreasing by Proposition~\ref{decreasing1}, we
know that the universal upper bound is $f_{2,2}(q,4)\leq m$ when $m=k(m)$. The
solution is $m=2q-1$, which is achieved when $h=1$ and $r=q-1$.
Hence, we have the following upper bound for $f_{2,2}(q,4)$.

\begin{proposition}\label{f22}
$f_{2,2}(q,4)\leq 2q-1$.
\end{proposition}

Next, we will give a construction of an Armstrong instance of $2q-1$
columns over $q$ symbols for $K_{2q-1}^4$ being the system of
minimal $(2,2)$-dependent keys. The construction is based on the
classical near $1$-factorization of complete graphs.


Let $n=2q-1$ and $K_n$ be a complete graph with vertex set $\bz_n$.
For each $i\in \bz_n$, take
\[T_i=\{\{t+i,-t+i\}: t\in [q-1]\},\]
 where the addition is in $\bz_n$. Then $\{T_i:i\in
\bz_n\}$ is a {\em near $1$-factorization of $K_n$}. Each $T_i$ is a
{\em near $1$-factor} which misses the point $i$. The following fact
is necessary for the construction of Armstrong code.

\vskip 10pt

\begin{proposition}
\label{factor}Let $n=2q-1$. For any distinct $i$, $j$, $k\in \bz_n$, there exist
three points $x$, $y$, $z\in \bz_n$, such that $\{x,y\}\in T_i$,
$\{y,z\}\in T_j$ and $\{z,x\}\in
T_k$. 
\end{proposition}

\begin{proof} First note the fact that for each $i\in
\bz_n$, an edge $\{x,y\}$ belongs to $T_i$ if and only if $x+y=2i$.
Suppose $j=s+i$ for some $s\in \bz_n\setminus\{0\}$. Then $T_i$ and
$T_j$ form a path from the vertex $j$ to $i$ as follows.
\begin{align*}
P=(&\{s+i,-s+i\},\{-s+i,3s+i\},\{3s+i,-3s+i\},\\&\ldots,\{(2q-3)s+i,-(2q-3)s+i\},\\
&\{-(2q-3)s+i,(2q-1)s+i\}).
\end{align*}
The last vertex is $i$ since $(2q-1)s+i=i$. Note that the edges in
$P$ are from $T_i$ and $T_j$ in turn and the length of $P$ is
$2(q-1)$. Let $T$ be the set of edges by joining two vertices in $P$
of distance two. We claim that $T\cap T_k\neq \emptyset$ for any
$k\neq i,j$.

By the observation at the beginning, we only need to prove that
there exists an edge $\{x,y\}$  in $T$ such that $x+y=2k$. Let
$S=\{x+y: \{x,y\}\in T\}$. By the form of $P$, we have
$S=\{4ts+2i:t\in [q-1]\}\cup\{-4ts+2i: t\in [q-2]\}$. It is easy to
check that elements in $S$ are all different, i.e., $|S|=2q-3$.
Further, $2i\not\in S$ since $s\neq 0$ and $t\neq 0$. Also, $2j \not
\in S$ since $4ts\neq 2s$ for all $t\in [q-1]$ and $-4ts\neq 2s$ for
all $t\in[q-2]$. Hence, $S=\bz_n\setminus \{2i,2j\}$, or for any
$k\neq i,j$, $2k\in S$. This completes the proof.
\end{proof}

\vskip 10pt

\begin{proposition}
\label{s22o} There exists a $(q,4,2q-1)_{2,2}$-Armstrong code for
each $q\geq 3$.
\end{proposition}

\begin{proof} Let $n=2q-1$. We construct an $n\times n$ array $C$ over $[q]$ as follows.

The columns of $C$ are indexed by $T_i$ ($i\in \bz_n$), while the
rows are indexed by the vertices of $K_n$, i.e., $\bz_n$. In each
column, say column indexed by $T_i$,  arbitrarily order the $q-1$
edges of $T_i$, assign symbols in $[q]$ to this column,  such that
for each row $s\in \bz_n\setminus \{i\}$, symbol $j$ is assigned if
and only if $s$ is incident to the $j$-th edge of $T_i$; for row
$i$, symbol $q$ is assigned to the column.

We claim that $C$ is a $(q,4,2q-1)_{2,2}$-Armstrong code. For each
three rows $x,y,z$ of $C$, we choose three columns $T_i, T_j, T_k$,
such that $\{x,y\} \in T_i$, $\{x,z\}\in T_j$ and $\{y,z\}\in T_k$.
Then these three columns have exactly two distinct symbols in rows
$x,y,z$. For any three columns $T_i, T_j, T_k$, by
Proposition~\ref{factor}, we have three rows $x,y,z$ such that
$\{x,y\} \in T_i$, $\{x,z\}\in T_j$ and $\{y,z\}\in T_k$, i.e., at
most two distinct elements in these three rows. Further, since
$\{T_i:i\in \bz_n\}$ is a near $1$-factorization, no pairs from
$\{x,y,z\}$ occur in any $T_l$ with $l\neq i,j,k$. That is for any
other column $T_l$, there are exactly three different elements in
rows $x,y,z$. Thus we prove the claim.\end{proof}

Here is an example of applying Proposition~\ref{s22o} to a near
$1$-factorization of $K_n$.

\begin{example}
For $n=7$, we can get a near $1$-factorization of $K_7$ as follows:
\[\begin{array}{l}
T_0=\{\{1,6\},\{2,5\},\{3,4\}\}, \\T_1=\{\{2,0\},\{3,6\},\{4,5\}\}, \\T_2=\{\{3,1\},\{4,0\},\{5,6\}\}, \\T_3=\{\{4,2\},\{5,1\},\{6,0\}\}, \\ T_4=\{\{5,3\},\{6,2\},\{0,1\}\}, \\T_5=\{\{6,4\},\{0,3\},\{1,2\}\}, \\T_6=\{\{0,5\},\{1,4\},\{2,3\}\}.\\
\end{array}\]
Then a $(4,4,7)_{2,2}$-Armstrong code is constructed as below.

$$\begin{array}{ccccccc}
4 & 1 & 2 & 3 & 3 & 2 & 1 \\
1 & 4 & 1 & 2 & 3 & 3 & 2 \\
2 & 1 & 4 & 1 & 2 & 3 & 3 \\
3 & 2 & 1 & 4 & 1 & 2 & 3 \\
3 & 3 & 2 & 1 & 4 & 1 & 2 \\
2 & 3 & 3 & 2 & 1 & 4 & 1 \\
1 & 2 & 3 & 3 & 2 & 1 & 4 \\
\end{array}$$
\end{example}

\vskip 10pt

Combining Propositions~\ref{f22} and~\ref{s22o}, we determine the
value of $f_{2,2}(q,4)$. \vskip 10pt

\begin{theorem}
\label{f22s} $f_{2,2}(q,4)=2q-1$ for all integers $q\geq 3$.
\end{theorem}

\vskip 10pt

\section{Lower Bounds for $f_{1,t}(q,k)$}
Since each $(q,k,n)_{1,t}$-Armstrong code is trivially a
$(q,k,n)_{s,t}$-Armstrong code, $1<s\leq t$, the problem of
estimating values of $f_{1,t}(q,k)$ seems more important. In this
section, we focus on exploring  lower bounds for $f_{s,t}(q,k)$ when
$s=1$.

\subsection{A Construction from Reed-Solomon Codes}
In this subsection, we assume that $q$ is a prime power. Let $\gf_q$
be a finite field with $q$ elements $a_1,a_2,\ldots,a_q$. For each
polynomial $f\in \gf_q[X]$, let $f_{\infty}$ denote the coefficient
of $X$ in $f$. A Reed-Solomon code $C$ over $\gf_q$ of length $q+1$
is constructed as follows.
\[C=\{(f(a_1),f(a_2),\ldots,f(a_q),f_{\infty}):f\in \gf_q[X],\deg f<k\}.\]
 \vskip 10pt

\begin{proposition}\label{rs}
The code $C$ is a $(q,k,q+1)_{1,t}$-Armstrong code for any $1\leq
t\leq q-1$. Thus $f_{1,t}(q,k)\geq q+1$ for $1\leq t\leq q-1$ and $q$ a prime power.
\end{proposition}
\begin{proof} We prove it by definition of Armstrong code. View $C$
 as a $q^{k} \times (q+1)$ array with columns indexed by $(a_1,\ldots,a_q,\infty)$
  and rows indexed by polynomials in $\gf_q[X]$. Since $\deg f<k$, any $k$
coordinates determine a unique polynomial $f$. Hence, any two
codewords agree in at most $k-1$ positions, which means that the first
condition of the definition holds. As for the second condition,
choose any $k-1$ columns, say $a_{i_1},\ldots,a_{i_{k-1}}$ and any
$k-1$ elements $b_1,\ldots,b_{k-1}$ in $\gf_q$, then there are
exactly $q$ polynomials $f_l\in \gf_q[X]$, $l\in [q]$ such that
$f_l(a_{i_j})=b_j$, $j\in [k-1]$. Since any two codewords agree at
most $k-1$ positions, then any columns outside
$\{a_{i_1},\ldots,a_{i_{k-1}}\}$ have exactly $q$
distinct elements in the $q$ rows $f_l$, $l\in [q]$.
\end{proof} \vskip 10pt

\subsection{An Existence Result Using the Probabilistic Method}

In this subsection, we will  give a lower bound for $f_{1,t}(q,k)$
by using a similar probabilistic method as in
\cite{SaliSzekely:2008}. First, we construct a random $q$-ary code
$C$ of length $n$ and size $(t+1)\cdot {n\choose {k-1}}$ as follows.

For each subset of $k-1$ positions $K\subset [n]$, choose a set of
$t+1$ codewords $A^K=\{A_1^K,A_2^K,\dots,A_{t+1}^K\}$ randomly such
that they pairwise agree exactly at the positions in $K$. That is,
in each position in $K$, a random symbol is chosen with probability
$\frac{1}{q}$ and assigned to this position for all codewords in $A^K$. In each position out of
$K$, $t+1$ distinct symbols are randomly chosen and assigned to the $t+1$
rows. The choices are pairwise independent for
distinct positions. Let $C=\cup_{K\subset [n]} A^K$. Then the choice
of $C$ makes it satisfy the second property of a
$(q,k,n)_{1,t}$-Armstrong code. Next, we will prove that $C$ also
satisfies the first property with positive probability under certain
conditions.

Consider events $v(A_i^K,A_j^L)$, where $i,j\in [t+1]$ and $K\neq L$
are $(k-1)$-subsets of coordinate positions, that the two codewords
agree in at least $k$ coordinates. Two such events $v(A_i^K,A_j^L)$
and $v(A_{i'}^{K'},A_{j'}^{L'})$ are independent if
$\{K,L\}\cap\{K',L'\}=\emptyset$. If for any two distinct
$(k-1)$-subsets $K$, $L$ and any pair $i,j\in [t+1]$, event
$v(A_i^K,A_j^L)$ doesn't happen, then $C$
 satisfies the first condition, i.e. $C$ is a
$(q,k,n)_{1,t}$-Armstrong code.

Define the dependency graph $G=(V,E)$ by $V$ being the set of events
$\{v(A_i^K,A_j^L): K\neq L, i,j\in [t+1]\}$, and $v(A_i^K,A_j^L)$
and $v(A_{i'}^{K'},A_{j'}^{L'})$ are connected by an edge if and
only if $\{K,L\}\cap\{K',L'\}\neq\emptyset$. Thus the degree of
$v(A_i^K,A_j^L)$ in the dependency graph is $2(t+1)^2{n\choose
k-1}-3(t+1)^2-1$. On the other hand,
$$\pr(v(A_i^K,A_j^L))=\sum_{l=k}^{n}{n\choose
l}(\frac{1}{q})^l(1-\frac{1}{q})^{n-l}=B(k,n,\frac{1}{q}).$$ By the
well-known Chernoff bound, \[B(k,n,\frac{1}{q})\leq
(\frac{n}{qk})^k\cdot e^{k-\frac{n}{q}},\] when $k>\frac{n}{q}$.

Now, we will apply the following famous Lov\'asz' Local Lemma to
give a lower bound for $f_{1,t}(q,k)$.

 \vskip 10pt

\begin{lemma}[\cite{AS2002}]\label{LLL}
Let $A_1,A_2,\dots,A_n$ be events in an arbitrary probability space.
Suppose that each event $A_i$ is mutually independent of a set of
all the other events $A_j$ but at most $d$, and that $\pr(A_i)\leq
p$ for all $1\leq i\leq n$. If $ep(d+1)\leq 1$, then
$\pr(\bigwedge_{i=1}^{n}\overline{A_i})>0$.
\end{lemma}

\vskip 10pt

By Lemma~\ref{LLL}, if
\begin{equation}\label{ncine}2(t+1)^2{n\choose
k-1}B(k,n,\frac{1}{q})<\frac{1}{e},\end{equation}
 then
$$\pr(\bigcap \overline{v(A_i^K,A_j^L)})>0,$$
which means that a $(q,k,n)_{1,t}$-Armstrong code exists with
positive probability.

Assuming $n> 2k$, (\ref{ncine}) follows from
\begin{equation*}2(t+1)^2{n\choose
k}B(k,n,\frac{1}{q})<\frac{1}{e}.\end{equation*} Since ${n\choose
k}\leq (\frac{n\cdot e}{k})^k$, it is enough to show that
\begin{equation*}(\frac{n\cdot e}{k})^k (\frac{n}{qk})^k\cdot
e^{k-\frac{n}{q}}<\frac{1}{2e(t+1)^2}.
\end{equation*}
Writing $n=ck$, we have
\begin{equation}\label{nc2}\frac{c^2}{q} \cdot
e^{2-\frac{c}{q}}<\sqrt[k]{\frac{1}{2e(t+1)^2}}.
\end{equation}
It is clear that (\ref{nc2}) is true when $c\leq
\sqrt[2k]{\frac{1}{2e(t+1)^2}} \frac{\sqrt{q}}{e}$.
 \vskip 10pt

\begin{proposition}\label{probresult}
Assume that $t\geq 1$ and $q,k$ are integers satisfying
 $q> 4e^2\sqrt[k]{2e(t+1)^2}$. Then a
$(q,k,n)_{1,t}$-Armstrong code exists for  $n=
\sqrt[2k]{\frac{1}{2e(t+1)^2}} \frac{\sqrt{q}}{e} k$, i.e.,
\[f_{1,t}(q,k)\geq \sqrt[2k]{\frac{1}{2e(t+1)^2}} \frac{\sqrt{q}}{e}
k.\]
\end{proposition}
\begin{proof} The condition $q> 4e^2\sqrt[k]{2e(t+1)^2}$ implies
that $n> 2k$ and $qk>n$, which completes the proof by combining
above analysis.
\end{proof} \vskip 10pt

\section{Conclusion}

We investigated the maximum number of minimal keys in relational
database systems with attributes having bounded domains via
the study of Armstrong codes. We showed that the maximum
length $n$ for which a $(q,3,n)$-Armstrong code can exist
is $f(q,3)=3q-1$ for all $q\geq
5$ with three possible exceptions, disproving a conjecture of Sali.

Our determination of $f(q,3)$ involves introducing the new concept
of extorthogonal double covers (extODC), a generalization of
orthogonal double covers with property that any two partitions cover
at least one common $2$-subset. This new combinatorial design is
interesting not only in database theory, but also in design theory.
Similar to ODCs, there are several directions for the study of
extODCs. For example, each partition could be extended to any
spanning subgraph, or consider similar properties for hypergraphs.

Further, we generalized Armstrong codes to the case of
$(s,t)$-dependencies. The maximum length $n=f_{s,t}(q,k)$ for which
a $(q,k,n)_{s,t}$-Armstrong code can exist seems to be quite
difficult to determine. Classes of optimal Armstrong codes of this
type are constructed. Several lower bounds of $f_{s,t}(q,k)$ are
also established.

%
\section{Acknowledgement}
The authors would like to thank Professor Tuvi Etzion for pointing
out the construction from Reed-Solomon codes. The authors are also grateful to the anonymous referees and
Professor Dr. Navin Kashyap for their invaluable and
constructive comments and suggestions which have greatly
improved the presentation of this paper and make
this paper more readable.



\begin{IEEEbiographynophoto}{Yeow Meng Chee}
(SM'08) received the B.Math. degree in computer science
and combinatorics and optimization and the M.Math. and Ph.D. degrees in
computer science from the University of Waterloo, Waterloo, ON, Canada, in
1988, 1989, and 1996, respectively.

Currently, he is a Professor at the Division of Mathematical Sciences,
School of Physical and Mathematical Sciences, Nanyang Technological University,
Singapore. Prior to this, he was Program Director of Interactive Digital
Media R\&D in the Media Development Authority of Singapore, Postdoctoral
Fellow at the University of Waterloo and IBM¡¯s Z\"urich Research Laboratory,
General Manager of the Singapore Computer Emergency Response Team,
and Deputy Director of Strategic Programs at the Infocomm Development
Authority, Singapore.

His research interest lies in the interplay between combinatorics and computer
science/engineering, particularly combinatorial design theory, coding
theory, extremal set systems, and electronic design automation.
\end{IEEEbiographynophoto}

\begin{IEEEbiographynophoto}{Hui Zhang}
 received the Ph.D. degree in Applied Mathematics from Zhejiang University, Hangzhou, Zhejiang, P. R. China in 2013. Currently, she is working as a Research Fellow in Nanyang Technological University in Singapore. Her research interests include combinatorial theory, coding theory and cryptography, and their intersections.

\end{IEEEbiographynophoto}

\begin{IEEEbiographynophoto}{Xiande Zhang}
received the Ph.D. degree in mathematics from Zhejiang University, Hangzhou, Zhejiang, P. R. China in 2009. After that, she held postdoctoral positions in Nanyang Technological University and Monash University. Currently, she is a Research Fellow at the Division of Mathematical Sciences, School of Physical and Mathematical Sciences, Nanyang Technological University, Singapore.  Her research interests include combinatorial design theory, coding theory, cryptography, and their interactions.
\end{IEEEbiographynophoto}
\end{document}